\documentclass{amsart}

\newcommand{\C}{\mathcal C}
\newcommand{\U}{\mathcal U}
\newcommand{\V}{\mathcal V}
\newcommand{\IR}{\mathbb R}
\newcommand{\w}{\omega}
\newcommand{\WAP}{\mathrm{WAP}}
\newcommand{\AP}{\mathrm{AP}}
\newcommand{\SAP}{\mathrm{SAP}}
\newcommand{\dist}{\mathrm{dist}}
\newcommand{\pr}{\mathrm{pr}}
\newcommand{\Ra}{\Rightarrow}

\newtheorem{theorem}{Theorem}[section]
\newtheorem{proposition}[theorem]{Proposition}
\newtheorem{lemma}[theorem]{Lemma}
\newtheorem{example}[theorem]{Example}
\newtheorem{problem}[theorem]{Problem}

\title[Compact extensions of topological semigroups]{Openly factorizable spaces and compact extensions of topological semigroups}

\author{Taras Banakh}
\address{Instytut Matematyki, Akademia
\'Swi\c etokrzyska, Kielce, Poland \\ and Department of
Mathematics, Lviv National University, \\ Universytetska 1, 79000,
Ukraine} \email{tbanakh@yahoo.com}
\author{Svetlana Dimitrova}
\address{National Technical University ``Kharkiv Polytechnical Institute", Frunze 21, Kharkiv, 61002, Ukraine}
\email{s.dimitrova@mail.ru}
\subjclass{22A15; 54B30; 54C20; 54C08; 54D35}
\keywords{topological semigroup, semigroup compactification, inverse spectrum, opseudocompact space, openly factorizable space, openly generated space, Eberlein compact, Corson compact, Valdivia compact}

\begin{document}

\begin{abstract} We prove that the semigroup operation of a topological semigroup $S$ extends to a continuous semigroup operation on its the Stone-\v Cech compactification $\beta S$ provided $S$ is a pseudocompact openly factorizable space, which means that each map $f:S\to Y$ to a second countable space $Y$ can be written as the composition $f=g\circ p$ of an open map $p:X\to Z$ onto a second countable space $Z$ and a map $g:Z\to Y$. We present a spectral characterization of openly factorizable spaces and establish some properties of such spaces.
\end{abstract}
\maketitle

This paper was motivated by the problem of detecting topological semigroups that embed into compact topological semigroups. One of the ways to attack this problem is to find conditions on a topological semigroup $S$ guaranteeing that the semigroup operation of $S$ extends to a continuous semigroup operation on the Stone-\v Cech compactification $\beta S$ of $S$. A crucial step in this direction was made by E.Reznichenko \cite{Rez} who proved that the semigroup operation on a pseudocompact topological semigroup $S$ extends to a separately continuous semigroup operation on $\beta S$. In this paper we show that the extended operation on $\beta S$ is continuous if the space $S$ is separable and openly factorizable, which means that each continuous map $f:S\to Y$ to a second countable space $Y$ can be written as the composition $f=g\circ p$ of an open continuous map $p:X\to Z$ onto a second countable space $Z$ and a continuous map $g:Z\to Y$. The class of openly factorizable spaces turned to be interesting by its own so we devote Sections~\ref{s2}, \ref{s3} to studying such spaces.

We recall that the {\em Stone-\v Cech compactification} of a Tychonov space $X$ is a
compact Hausdorff space $\beta X$ containing $X$ as a dense subspace so that each continuous map $f:X\to Y$ to a compact Hausdorff space $Y$ extends to a continuous map $\bar f:\beta X\to Y$. 

Replacing in this definition compact Hausdorff spaces by real complete spaces we obtain the definition of the Hewitt completion $\upsilon X$ of $X$. We recall that a topological space $X$ is {\em  real complete} if $X$ is homeomorphic to a closed subspace of some power $\IR^\kappa$ of the real line. Thus a {\em Hewitt completion} of a Tychonov space $X$ is a real complete space $\upsilon X$ containing $X$ as a dense subspace so that each continuous map $f:X\to Y$ to a real complete space $Y$ extends to a continuous map $\upsilon f:\upsilon X\to Y$. By \cite[3.11.16]{En}, the Hewitt completion $\upsilon X$ can be identified with the subspace 
$$\{x\in\beta X:G\cap X\ne\emptyset\mbox{ for any $G_\delta$-set $G\subset\beta X$ with $x\in G$}\}$$of the Stone-\v Cech compactification $\beta X$ of $X$.

The Hewitt completion $\upsilon X$ of a Tychonov space $X$ coincides with its Stone-\v Cech compactification $\beta X$ if and only if the space $X$ is {\em pseudocompact} in the sense that each continuous real-valued function on $X$ is bounded, see \cite[\S3.11]{En}. On the other hand, if a Tychonov space $Z$ is real complete, then $\upsilon Z=Z$, see  \cite[3.11.12]{En}.

The problem of extending the group operation from a (para)topological group $G$ to its Stone-\v Cech or Hewitt extensions have been considered in \cite{Rez}, \cite{AH}, \cite{PT}, \cite{RU}. In this paper we address a similar problem for topological semigroups.  
All topological spaces appearing in this paper are Tychonov.

\section{Semigroup compactifications of topological semigroups}

In this section we recall some information on semigroup compactifications of a given (semi)topological semigroup $S$.

By a {\em semitopological semigroup} we understand a topological space $S$ endowed with a separately continuous semigroup operation $*:S\times S\to S$. If the operation is jointly continuous, then $S$ is called a {\em topological semigroup}.

Let $\C$ be a class of compact Hausdorff semitopological semigroups. By a {\em $\C$-compactification} of a semitopological semigroup $S$ we understand a pair $(\C(S),\eta)$ consisting of a compact semitopological semigroup $\C(S)\in\C$ and a continuous homomorphism $\eta:S\to\C(S)$ (called the {\em canonic homomorphism}) such that for each continuous homomorphism $h:S\to K$ to a semitopological semigroup $K\in\C$ there is a unique continuous homomorphism $\bar h:\C(S)\to K$ such that $h=\bar h\circ\eta$. It follows that any two $\C$-compactifications of $S$ are topologically isomorphic. We shall be interested in $\C$-compactifications for the following classes of semigroups:
\begin{itemize}
\item $\WAP$ of compact semitopological semigroups; 
\item $\AP$ of compact topological semigroups;
\item $\SAP$ of compact topological groups.
\end{itemize}
The corresponding $\C$-compactifications of a semitopological semigroup $S$ will be denoted by $\WAP(S)$, $\AP(S)$, and $\SAP(S)$. The notation came from the abbreviations for weakly almost periodic, almost periodic, and strongly almost periodic function rings that determine those compactifications, see \cite[\S III.2]{Rup}.

The inclusions of the classes $\SAP\subset\AP\subset\WAP$ induce canonical continuous homomorphisms 
$$\eta:S\to\WAP(S)\to\AP(S)\to\SAP(S)$$
for each semitopological semigroup $S$.
Since the space $\WAP(S)$ is compact, the canonical map $\eta:S\to\WAP(S)$ uniquely extends to a continuous map $\beta\eta:\beta S\to\WAP(S)$ defined on the Stone-\v Cech compactification of $S$.

It should be mentioned that the canonic homomorphism $\eta:S\to \WAP(S)$ needs not be injective. For example, for the group $H_+[0,1]$ of orientation-preserving homeomorphisms of the interval the WAP-compactification is a singleton, see \cite{Megrel}.  However, for pseudocompact semitopological semigroups the situation is more optimistic. The following two results are due to E.Reznichenko \cite{Rez}. They allow us to identify the WAP-compactification $\WAP(S)$ of a (countably compact) pseudocompact topological (semi)semigroup $S$ with the Stone-\v Cech compactification $\beta S$ of $S$. We recall that a topological space $X$ is {\em countably compact} if each countable open cover of $X$ has a finite subcover.

\begin{theorem}[Reznichenko]\label{rez1} For any countably compact semitopological semigroup $S$ the semigroup operation $S\times S\to S$ extends to a separately continuous semigroup operation $\beta S\times\beta S\to\beta S$, which implies that the canonic map $\beta\eta:\beta S\to\WAP(S)$ is a homeomorphism.
\end{theorem}

The same conclusion holds for  pseudocompact topological semigroups.

\begin{theorem}[Reznichenko]\label{rez2} For any pseudocompact topological semigroup $S$ the semigroup operation $S\times S\to S$ extends to a separately continuous semigroup operation $\beta S\times\beta S\to\beta S$, which implies that the canonic map $\beta\eta:\beta S\to\WAP(S)$ is a homeomorphism.
\end{theorem}

If a topological semigroup $S$ has pseudocompact square, then its Stone-\v Cech compactification $\beta S$ coincides with its $\AP$-compactification.

\begin{theorem} For any topological semigroup $S$ with pseudocompact square $S\times S$ the semigroup operation $S\times S\to S$ extends to a continuous semigroup operation $\beta S\times\beta S\to\beta S$, which implies that the canonic maps $\beta S\to\WAP(S)\to\AP(S)$ are homeomorphisms.
\end{theorem}

\begin{proof} By Theorem~\ref{rez2}, the semigroup operation $\mu:S\times S\to S$ of $S$ extends to a separately continuous semigroup operation $\overline{\mu}:\beta S\times\beta S\to\beta S$ on  $\beta S$. On the other hand, the operation $\mu:S\times S\to S\subset\beta S$ extends to a continuous map $\beta\mu:\beta(S\times S)\to\beta S\times\beta S$. By the Glicksberg Theorem \cite[3.12.20(c)]{En}, the pseudocompactness of the square $S\times S$ implies that the Stone-\v Cech extension $\beta i:\beta(S\times S)\to\beta S\times\beta S$ of the inclusion map $i:S\times S\to\beta S\times\beta S$ is a homeomorphism. Observe that the maps $\beta \mu$ and $\bar\mu\circ\beta i$ coincide on the dense subset $S\times S$ of $\beta S\times\beta S$. It is an easy exersice to check that those maps coincide everywhere, which implies that the map $\bar\mu=\beta\mu\circ (\beta i)^{-1}$ is continuous. This means that $\beta S$ is a compact topological semigroup and hence the canonic map $\beta\eta:\beta S\to\AP(S)$ has continuous inverse.
\end{proof}

It should be mentioned that for a pseudocompact topological semigroup $S$ the canonic map $\eta:S\to\AP(S)$ needs not be a topological embedding.  The following counterexample is constructed in \cite{BDG}. 

\begin{example}\label{ex1} If there is a Tkachenko-Tomita group, then there is a countably compact topological semigroup $S$ for which the canonic homomorphism $\eta:S\to\AP(S)$ is not injective.
\end{example}

By a Tkachenko-Tomita group we understand a commutative torsion-free countably compact topological group without non-trivial convergent sequences. The first example of such a group was constructed by M.Tkachenko \cite{Tka} under the Continuum Hypothesis, which was later weakened to some forms of the Martin Axiom by A.Tomita et al. \cite{Tom96},
\cite{KTW}, \cite{GFTW}, \cite{MGT}. We do not know if a Tkachenko-Tomita group exists in ZFC.

Example~\ref{ex1} shows that one should impose rather strong restrictions on a topological semigroup $S$ to guarantee that the canonic homomorphism $S\to\AP(S)$ (or $S\to\SAP(S)$) is an embedding. 

Observe that for every semitopological semigroup $S$ its SAP-compactification $\SAP(S)$ is a compact topological group. It is well-known that a semitopological semigroup $S$ is topologically isomorphic to a subgroup of a compact topological group if and if $S$ is a totally bounded topological group. We recall that a topological group $G$ is called {\em totally bounded} if for every non-empty open subset $U\subset G$ there is a finite subset $F\subset G$ such that $G=FU=UF$. 

The following important result can be found in \cite[III.3.3]{Rup}.

\begin{theorem}[Ruppert]\label{rup} For each totally bounded topological group $G$ the canonic homomorphisms $\WAP(G)\to\AP(G)\to\SAP(G)$ are homeomorphisms and the canonic map $\eta:G\to \SAP(G)$ is a topological embedding.
\end{theorem}

The same conclusion holds for Tychonov pseudocompact topological semigroups that contain dense totally bounded topological subgroups.

\begin{theorem}\label{totalbound} If a pseudocompact topological semigroup $S$ contains a totally bounded topological group $H$ as a dense subgroup, then the canonic maps $$\beta S\to\WAP(S)\to\AP(S)\to\SAP(S)$$ are homeomorphisms.
\end{theorem}

\begin{proof}  The embedding $H\subset S$ induces a continuous homomorphism $h:\WAP(H)\to\WAP(S)$. We claim that this homomorphism is surjective. Indeed, 
 by Theorem~\ref{rup}, $\WAP(H)$ is a compact topological group, containing $H$ as a dense subgroup.  By Theorem~\ref{rez2}, the Stone-\v Cech compactification $\beta S$ of $S$ can be identified with the WAP-compactification $\WAP(S)$ of $S$. Then the image $h(\WAP(H))$ contains the dense subset $H$ of $\beta S=\WAP(S)$ and hence coincides with $\beta S$ being a compact dense subset of $\beta S$. The compact semitopological semigroup $\WAP(S)$, being a continuous homomorphic image of the compact topological group $\WAP(H)$, is a compact topological group. This implies that the canonic homomorphism $\WAP(S)\to\SAP(S)$ is a topological isomorphism. Consequently, the maps $\beta S\to\WAP(S)\to\AP(S)\to\SAP(P)$ all are homeomorphisms.
\end{proof}

This theorem implies another one of the same spirit.

\begin{theorem} If a topological semigroup $S$ contains a dense subgroup and has countably compact square $S\times S$, then the canonic maps
$\beta S\to\WAP(S)\to\AP(S)\to \SAP(S)$ are homeomorphisms.
\end{theorem}

\begin{proof} Let $H$ be a dense subgroup of $S$ and let $e$ be the idempotent of $H$.
Let $$H_e=\{x\in S:\exists x^{-1}\in S\;\mbox{with}\;xx^{-1}=x^{-1}x=e,\; xe=ex=x,\;x^{-1}e=ex^{-1}=x^{-1}\}$$be the maximal subgroup of $S$ containing the idempotent $e$. Our theorem will follow from Theorem~\ref{totalbound} as soon as we check that $H_e$ is a totally bounded topological group. For this observe that $H_e$ coincides with the projection of the closed subset $$A=\{(x,y)\in S\times S:xy=yx=e,\; xe=ex=x,\;ye=ey=y\}$$ of $S\times  S$ onto the first factor. The countable compactness of $S\times S$ implies that of $A$ and of its projection $H_e$. The paratopological group $H_e$, being a Tychonov countably compact paratopological group, is a totally bounded topological group according to \cite[2.7]{Rez}.
\end{proof}

Our final result concerns the $\AP$-compactifications of pseudocompact openly factorizable topological semigroups. Those are pseudocompact topological semigroups whose topological spaces are {\em openly factorizable}.

We define a topological space $X$ to be {\em openly factorizable} if for each continuous map $f:X\to Y$ to a second countable space $Y$ there are a continuous open map $p:X\to Z$ onto a second countable space $Z$ and a continuous map $g:Z\to Y$ such that $f=g\circ p$. Openly factorizable spaces will be studied in details in the next two sections. Now we present our main extension result for which we need the notion of a weakly Lindel\"of space.

We call a topological space $X$  {\em weakly Lindel\"of\/} if each open cover $\U$ of $X$ contains a countable subcollection $\V\subset\U$ whose union $\cup\V$ is dense in $X$. It is clear that the class of weakly Lindel\"of spaces includes all Lindel\"of spaces and all countably cellular (in particular, all separable) spaces. 

\begin{theorem}\label{main} For any openly factorizable topological semigroup $S$ having weakly Lindel\"of square $S\times S$, the semigroup operation $S\times S\to S$  extends to a continuous semigroup operation $\upsilon S\times\upsilon X\to\upsilon S$ defined on the Hewitt completion $\upsilon S$ of $S$.
\end{theorem}

\begin{proof} By Theorem~\ref{product} below the semigroup operation $\mu:S\times S\to S$ extends to a continuous map $\bar\mu:\upsilon S\times\upsilon S\to\upsilon S$ thought as a continuous binary operation on $\upsilon S$. This operation is associative on $S$ and by the continuity remains associative on $\upsilon S$.
\end{proof} 

This theorem implies another one:

\begin{theorem} For each pseudocompact openly factorizable topological semigroup $S$ with weakly Lindel\"of square the canonic maps $\beta S\to\WAP(S)\to\AP(S)$ are homeomorphisms.
\end{theorem}

\begin{proof} By Theorem~\ref{main}, the semigroup operation $\mu:S\times S\to S$  extends to a continuous semigroup operation $\bar \mu:\upsilon S\times\upsilon S\to\upsilon S$ turning the Hewitt completion $\upsilon S$ of $S$ into a topological semigroup that contains $S$ as a dense subsemigroup. Since the space $S$ is pseudocompact, its Hewitt completion coincides with its Stone-\v Cech compactification $\beta S$ \cite[\S3.11]{En}. Consequently, $\beta S$ is a compact topological semigroup, which implies that the canonic map $\beta\eta:\beta S\to\AP(S)$ has a continuous inverse and consequently, the maps $$\beta S\to\WAP(S)\to\AP(S)$$ are homeomorphisms.
\end{proof}

\section{Some elementary properties of openly factorizable spaces}\label{s2}

In this section we establish some elementary properties of openly factorizable spaces. 
First we prove a helpful lemma.

\begin{lemma}\label{l1} Let $p:X\to Z$ be a map from a Tychonov space to a second countable space and let $\upsilon p:\upsilon X\to Z$ be its continuous extension to the Hewitt completion of $X$. The map $\upsilon p$ is surjective (open) if and only if so is the map $p$.
\end{lemma}

\begin{proof} Endow the second countable space $Z$ with a metric generating the topology of $Z$.

If the map $p$ is surjective, then $\upsilon p$ is surjective too because $$Z=p(X)\subset\upsilon p(\upsilon X)\subset Z.$$ Now assume conversely that the map $\upsilon p$ is surjective but $p$ is not. Then  
we can find a point $z_0\in Z\setminus p(X)$ and consider the continuous function $f:\upsilon X\to [0,+\infty)$, $f:x\mapsto\dist(p(x),z_0)$. It follows from $z_0\notin p(X)$ that $f(X)\subset(0,+\infty)$. The function $f|X:X\to(0,+\infty)$ has a unique continuous extension $\bar f:\upsilon  X\to (0,\infty)$. Since $f$ also extends $f|X$, we get $\bar f=f$ and hence $f(\upsilon X)=\bar  f(\upsilon X)\subset(0,\infty)$ which is not possible because $f(x_0)=0$ for any point $x_0\in p^{-1}(z_0)$. Hence the map $p|X:X\to Z$ is surjective.
\smallskip

Now assume that the map $p$ is open. To show that the map $\upsilon p$ is open, take  any open subset $U\subset \upsilon X$. We claim that $\upsilon p(U)=p(U\cap X)$. In the opposite case, we can find a point $y\in \upsilon p(U)\setminus p(U\cap X)$. Choose any point $x_0\in U$ with $\upsilon p(x_0)=y$ and find a  continuous function $g:\upsilon X\to[0,1]$ such that $g^{-1}(0)$ is a neighborhood of $x_0$ and $g^{-1}[0,1)\subset U$. Consider the continuous function $f:\upsilon X\to [0,\infty)$ defined by $f(x)=g(x)+\dist(\upsilon p(x),y)$ and note that $f(x_0)=0$ while $f(x)\in (0,1]$ for all $x\in X$. Indeed, if $x\in X\cap U$, then $f(x)\ge\dist(p(x),y)>0$ because $y\notin p(U\cap X)$. If $x\in X\setminus U$, then $f(x)\ge g(x)=1>0$.
Since $\upsilon X$ is a Hewitt completion of $X$, the function $f|X:X\to(0,+\infty)$ admits a unique continuous extension $\bar f:\upsilon X\to(0,+\infty)$. Since $X$ is dense in $\upsilon X$, we get $f=\bar f$ and thus $0=f(x_0)=\bar f(x_0)\in(0,\infty)$. This is a contradiction showing that the set $\upsilon p(U)=p(U\cap X)$ is open and hence the map $\upsilon p$ is open. 
\smallskip

Now assume that the map $\upsilon p$ is open. To show that $p$ is open, fix any non-empty open set $U\subset X$ and find an open set $V\subset\upsilon X$ such that $U=V\cap X$. To prove that the image $p(U)$ is open, take any point $y_0\in p(U)$ and find a point  $x_0\in U$ with $p(x_0)=y_0$. Since the space $\upsilon X$ is Tychonov, there is a continuous function $f:\upsilon X\to[0,1]$ such that $W=f^{-1}(0)$ is a neighborhood of $x_0$ in $\upsilon X$ while $f^{-1}[0,1)\subset V$. Since the map $\upsilon p$ is open, the image $\upsilon p(W)$ is an open neighborhood of $y_0$ in $Z$. We claim that $\upsilon p(W)\subset p(V\cap X)=p(U)$. Assume conversely that there is a point $y\in \upsilon p(W)\setminus p(U)$ and consider the continuous function 
$$g:\upsilon X\to[0,\infty),\;\; g(x)\mapsto f(x)+\dist(p(x),y).$$
It follows that $g(X)\subset(0,\infty)$ and hence $g(\upsilon X)\subset(0,\infty)$ too. On the other hand, for any point $x\in W$ with $p(x)=y$ we get $g(x)=0$, which is a contradiction showing that $p$ is open.
\end{proof}

\begin{proposition}\label{hfactor}
The Hewitt completion $\upsilon X$ of a Tychonov space $X$ is openly factorizable if and only if so is the space $X$.
\end{proposition}

\begin{proof} Assume that a Tychonov space $X$ is openly factorizable. To show that the Hewitt completion $\upsilon X$ is openly factorizable, take any continuous map $f:\upsilon X\to Y$ to a second countable space $Y$. Since $X$ is openly factorizable, there are an open surjective continuous map $p:X\to Z$ to a second countable space $Z$ and a continuous map $g:Z\to Y$ such that $f|X=g\circ p$. The space $Z$, being second countable, is real complete \cite[3.11.12]{En}. Consequently, the map $p$ admits a continuous extension $\upsilon p:\upsilon X\to Z$. It follows that $f=g\circ\upsilon p$.
By Lemma~\ref{l1}, the map $\upsilon p$ is open and surjective, witnessing that $\upsilon X$ is openly factorizable.

Now assume that $\upsilon X$ is openly factorizable. To show that $X$ is openly factorizable, take any continuous function $f:X\to Y$ to a second countable space $Y$. Since $Y$ is real complete \cite[3.11.12]{En},
the map $f$ extends to a continuous map $\upsilon f:\upsilon X\to Y$. Since $\upsilon X$ is openly factorizable, there are an open surjective continuous map $p:\upsilon X\to Z$ to a second countable space $Z$ and a continuous map $g:Z\to Y$ such that $f=g\circ p$. 
Then $f|X=g\circ p|X$ and the map $p|X:X\to Z$ is open and surjective by Lemma~\ref{l1}. 
\end{proof}

\begin{proposition}\label{bfactor} The Stone-\v Cech compactification $\beta X$ of a Tychonov space $X$ is openly factorizable if and only if $X$ is pseudocompact and openly factorizable.
\end{proposition}

\begin{proof} If $X$ is pseudocompact and openly factorizable, then the Hewitt completion $\upsilon X$ is openly factorizable by Proposition~\ref{hfactor}. Since $X$ is pseudocompact, its Hewitt completion coincides with the Stone-\v Cech compactification $\beta X$. So, $\beta X$ is openly factorizable.

Now assume conversely that $\beta X$ is openly factorizable. We claim that $X$ is pseudocompact. If the opposite case, we could find a continuous unbounded function $f:X\to[0,\infty)$. Let $\beta f:\beta X\to [0,\infty]$ be the Stone-\v Cech extension of the map $f$ to the one-point compactification of the half-line $[0,\infty)$.  Since $\beta X$ is openly factorizable, there are a continuous open surjective map $p:\beta X\to Z$ onto a metrizable compact space $Z$ and a continuous map $g:Z\to [0,\infty]$ such that $f=g\circ p$. 

Since the function $f$ is unbounded, we can choose a sequence $\{x_n\}_{n\in\w}\subset X$ such that the sequence $\{f(x_n)\}_{n\in\w}\subset[0,\infty)$ is strictly increasing and unbounded. Passing to a subsequence, if necessary, we can assume that the sequence $\{p(x_n)\}_{n\in\w}\subset Z$ converges to some point $z_\infty\in Z$. It follows from $f=g\circ p$ that $g(z_\infty)=\infty$ and the points $z_\infty$, $p(x_n)$, $n\in\w$, all are distinct. So each point $p(x_n)$ has a neighborhood $U_n\subset Z\setminus\{z_\infty\}$ such that the family $\{U_n:n\in\w\}$ is disjoint. Moreover, we can assume that the sequence $(U_n)$ converges to $z_\infty$ in the sense that each neighborhood $O(z_\infty)$ contains all but finitely many sets $U_n$. Since the sequence $\{f(x_n)\}_{n\in\w}$ is closed and discrete in $[0,\infty)$, to each point $f(x_n)$ we can assign an open neighborhood $V_n\subset [0,\infty)$ such that the family $\{V_n:n\in\w\}$ is discrete in $[0,\infty)$ (in the sense that each point  has a neighborhood that meets at most one set $V_n$). Now for every $n\in\w$ consider the open neighborhood $W_n=f^{-1}(V_n)\cap p^{-1}(U_n)$ of the point $x_n$ in $X$. Since the family $\{V_n\}_{n\in\w}$ is discrete in $[0,\infty)$, the family $\{W_n\}_{n\in\w}$ is discrete in $X$. Let $x_\infty\in\beta X$ be any accumulation point of the sequence $\{x_{2n}\}_{n\in\w}$.

Since the space $X$ is Tychonov and $\{W_{2n}\}_{n\in\w}$ is discrete, we can construct a continuous function $\varphi:X\to[0,1]$ such that 
$$\{x_{2n}\}_{n\in\w}\subset \varphi^{-1}(1)\subset \varphi^{-1}(0,1]\subset\bigcup_{n\in\w}W_{2n}.$$ 
Let $\beta\varphi :\beta X\to [0,1]$ be the Stone-\v Cech extension of $\varphi$. It follows from the continuity of $\beta\varphi$ that $\beta\varphi(x_\infty)=1$. Then the set $W=(\beta\varphi)^{-1}(\frac12,1]$ is an open neighborhood of $x_\infty$ in $\beta X$ with $$W\cap X\subset\overline{W}\cap X\subset \varphi^{-1}[1/2,1]\subset \bigcup_{n\in\w}W_{2n}.$$ It follows that $p(W\cap X)\subset V$ where $V=\bigcup_{n\in\w}V_{2n}$ and consequently,
$$p(W)\subset p(\overline{W})=p(\overline{W\cap X})\subset\overline{p(W\cap X)}\subset \overline{V}.$$ Since $\overline{V}\subset X\setminus \bigcup_{n\in\w}V_{2n+1}$ and $V_{2n+1}\to z_\infty$, the set $\overline{V}$ contains no neighborhood of the point $z_\infty=p(x_\infty)$. Consequently, the set $p(W)$ cannot be open. This contradiction completes the proof of the pseudocompactness of $X$.

In this case the Stone-\v Cech compactification $\beta X$ coincides with the Hewitt completion $\upsilon X$ of $X$. Applying Proposition~\ref{hfactor}, we conclude that $X$ is openly factorizable.
\end{proof}

\section{Spectral characterization of openly factorizable spaces}\label{s3}

In this section we shall present a spectral characterization of openly factorizable topological spaces. First we remind some information related to inverse spectra, see
\cite[\S3.1]{FC} and \cite[\S2.5]{En}. 

A partially ordered set $(A,\le)$ is called 
\begin{itemize}
\item {\em directed\/} if for every $a,b\in A$ there exists $c\in A$
with $c\ge a$, $c\ge b$; 
\item {\em $\w$-directed} if for any countable subset $C\subset A$ has an upper bound in $A$ (which a point $a\in A$ such that $a\ge c$ for every $c\in C$);
\item {\em $\omega$-complete\/} if $A$ each countable subset $C\subset A$ has the smallest upper bound $\sup C$ in $A$. 
\end{itemize}
For example, the ordinal $\omega_1$ endowed with the natural order is a well-ordered $\omega$-complete set.

By a spectrum over a directed set $(A,\le)$ we understand a
collection ${\mathcal S}=\{X_\alpha,\pi^\gamma_\alpha,A\}$
consisting of Tychonov spaces $X_\alpha$, $\alpha\in A$, and continuous surjective maps
$\pi_\alpha^\gamma:X_\gamma\to X_\alpha$ for $\alpha\le\gamma$
from $A$ such that $\pi_\alpha^\gamma=\pi_\alpha^\beta\circ
\pi_\beta^\gamma$ for every elements $\alpha\le\beta\le\gamma$ of
$A$. Let $$\lim {\mathcal S}=\{(x_\alpha)_{\alpha\in A}\in\prod_{\alpha\in A}X_\alpha:\forall\alpha,\beta\in A\;\; \alpha\le\beta\Rightarrow x_\alpha=\pi^\beta_\alpha(x_\beta)\}\subset \prod_{\alpha\in A}X_\alpha$$ denote the limit space of the
spectrum ${\mathcal S}$. 

For a directed subset $B$
of $A$ by ${\mathcal S}|B$ we denote the subspectrum ${\mathcal
S}|B=\{X_\alpha,\pi_\alpha^\gamma,B\}$ of ${\mathcal S}$,
consisting of the spaces $X_\alpha$ and the projections
$\pi_\alpha^\gamma$ for which $\alpha,\gamma\in B$. Given a
collection $\{f_\alpha:X\to X_\alpha\}_{\alpha\in A}$ of maps from a
space $X$ into the spaces of the spectrum ${\mathcal S}$ such that
$\pi_\alpha^\gamma\circ f_\gamma=f_\alpha$ for every
$\alpha\le\gamma$ in $A$ by $\lim f_\alpha:X\to\lim{\mathcal S}$
we denote the induced map into the limit space of ${\mathcal S}$.

A spectrum ${\mathcal S}=\{X_\alpha,\pi_\alpha^\gamma,A\}$ is
defined to be
\begin{itemize}
\item {\em continuous\/} if for every chain $B\subset A$ having
 supremum $\beta=\sup B$ the map $\lim\nolimits_{\alpha\in B}
\pi_\alpha^\beta:X_\beta\to\lim{\mathcal S}|B$ is a homeomorphism;
\item {\em open\/}
 if the projections $\pi_\alpha^\gamma:X_\gamma\to
X_\alpha$ are open and surjective for all $\alpha\le\gamma$ in
$A$;
\item {\em $\omega$-directed} (resp. {\em $\w$-complete}\/) provided so
is its index set $A$;
\item a {\em $\omega$-spectrum\/} if it is
$\omega$-directed and each space $X_\alpha$, $\alpha\in A$, is
second countable;
\item {\em factorizable\/} if every continuous map $f:\lim {\mathcal S}\to{\mathbb R}$ can be
written  as $f=f_\alpha\circ \pi_\alpha$ for some $\alpha\in A$
and some continuous map $f_\alpha:X_\alpha\to {\mathbb R}$.
\end{itemize}

According to \cite[3.1.5]{FC} a continuous $\omega$-complete spectrum
${\mathcal S}$ with surjective bonding maps is factorizable if and
only if every {\em bounded\/} continuous map $f:\lim{\mathcal
S}\to{\mathbb R}$ can be written as $f=f_\alpha\circ\pi_\alpha$
for some $\alpha\in A$ and some bounded continuous map
$f_\alpha:X_\alpha\to{\mathbb R}$. By another result of
\cite[3.1.7]{FC} a continuous $\omega$-complete open spectrum ${\mathcal
S}=\{X_\alpha,\pi_\alpha^\gamma, A\}$ is factorizable provided the limit space $\lim\mathcal S$ is countably cellular (that is contains no uncountable family of disjoint open sets). 

In fact, the proof of Proposition 3.1.7 of \cite{FC} can be modified to get the following more general statement, cf. \cite[3.2]{BCF}.

\begin{proposition}\label{factor} Suppose ${\mathcal S}=\{X_\alpha,\pi_\alpha^\gamma,A\}$
is a $\w$-spectrum and $X\subset\lim{\mathcal S}$ is
a weakly Lindel\"of subspace of its limit such that the
restrictions $\pi_\alpha|X:X\to X_\alpha$, $\alpha \in A$, of the limit projections
is open and surjective. Then every map
$f:X\to Y$ to a second countable space $Y$ can be written as
$f=f_\alpha\circ\pi_\alpha|X$ for some $\alpha\in A$ and some map
$f_\alpha:X_\alpha\to Y$. In particular, $X$ is
$C$-embedded into $\lim{\mathcal S}$ and hence $\lim{\mathcal S}$ is a Hewitt completion of $X$.
\end{proposition}

We recall that a subspace $X$ of a topological space $Y$ is {\em $C$-embedded} in $Y$ if each continuous functions $f:X\to \IR$ extends to a continuous function $\bar f:Y\to\IR$.

The following theorem gives a spectral characterization of openly factorizable spaces.

\begin{theorem}\label{spectral} A (weakly Lindel\"of) topological space $X$ is openly factorizable (if and) only if $X$ is a dense subspace of the limit space $\lim{\mathcal S}$ of an open $\w$-spectrum $\mathcal S=\{X_\alpha,\pi_\alpha^\gamma,A\}$ such that for every $\alpha\in A$ the restriction $\pi_\alpha|X:X\to X_\alpha$ of the limit projection is open and surjective.
\end{theorem}

\begin{proof} The ``if'' part follows immediately from Proposition~\ref{factor}. To prove the ``only if'' part, assume that a Tychonov space $X$ is openly factorizable. Let $A'$ be a set of all open continuous surjective maps $\alpha:X\to X_\alpha$ with $X_\alpha\subset\IR^\w$. The set $A$ is partially preordered by the relation: $\alpha\le\gamma$ if there is a continuous map $\pi^\gamma_\alpha:X_\gamma\to X_\alpha$ such that $\alpha=\pi^\gamma_\alpha\circ\gamma$. This map $\pi^\gamma_\alpha$ is necessarily open and surjective because the map $\alpha$ is open and surjective while $\gamma$ is continuous. Also the map $\pi^\gamma_\alpha$ is uniquely determined, which implies that $\pi^\gamma_\beta\circ \pi_\alpha^\beta=\pi^\gamma_\alpha$ for any $\alpha\le\beta\le\gamma$ in $A'$. This means that the relation $\le$ on $A'$ is transitive. The preorder $\le$ induces the equivalence relation $\cong$ on $A'$: $\alpha\cong\gamma$ if $\alpha\le\gamma$ and $\gamma\le\alpha$. Let $A$ be a subset of $A'$ intersecting each equivalence class in a single point. Then $A$ becomes a partially ordered set with respect to the order $\le$.

Let us show that the set $(A,\le)$ is $\w$-directed. Given a countable subset $C\subset 
A$ consider the diagonal product $f=\Delta_{\gamma\in C}\gamma:X\to \prod_{\gamma\in C}X_\gamma$. Taking into account that $\prod_{\gamma\in C}X_\gamma$ is second countable and $X$ is openly factorizable,  find an open surjective map $\alpha:X\to X_\alpha$ onto a second countable space $X_\alpha$ and a continuous map $g:X_\alpha\to \prod_{\gamma\in C}X_\gamma$ such that $g\circ \alpha=f$. We can assume that $X_\alpha\subset\IR^\w$
and thus $\alpha\in A'$. Moreover, we can replace $\alpha$ by an equivalent map and assume that $\alpha\in A$.
Let us show that $\alpha\ge\beta$ for each $\beta\in C$.
Consider the projection $\pr_\beta:\prod_{\gamma\in C}X_\gamma\to X_\beta$ and observe that the equality $g\circ\alpha=f$ implies $(\pr_\beta\circ g)\circ \alpha=\pr_\beta\circ f=\beta$, which means that $\alpha\ge\beta$.

Now we see that $\mathcal S=\{X_\alpha,\pi_\alpha^\gamma,A\}$ is an open $\w$-spectrum.
Let $\pi_\alpha:\lim\mathcal S\to X_\alpha$, $\alpha\in A$, be the limit projections of this spectrum. The open surjective maps $\alpha\in A$ determine a map 
$$A:X\to\lim{\mathcal S},\; A:x\mapsto (\alpha(x))_{\alpha\in A}$$ such that $p_\alpha\circ A=\alpha$ for every $\alpha\in A$. The surjectivity of the maps $\alpha\in A$ imply that the map $A:X\to\lim\mathcal S$ has dense image $A(X)\subset\lim\mathcal S$. Let us show that $A$ is a topological embedding. Given a point $x\in X$ and an open set $O(x)\subset X$ we should find an open set $U\subset\lim \mathcal S$ such that $A(x)\in U\cap A(X)\subset A(O(x))$. Since $X$ is Tychonov, there is a map $f:X\to[0,1]$ such that $x\in f^{-1}(0,1]\subset O(x)$.
The choice of the set $A$ guarantees that there is a map $\alpha:X\to X_\alpha$ in $A$ and a continuous map $g:X_\alpha\to (0,1]$ such that $g\circ \alpha=f$. Then the set $V=g^{-1}(0,1]$ is open in $X_\alpha$ and hence $U=p_\alpha^{-1}(V)$ is open in $\lim\mathcal S$. It is easy to check that this set $U$ has the required property: $A(x)\in U\cap A(X)\subset A(O(x))$.
\end{proof}

We apply the spectral characterization of openly factorizable spaces to derive the following main result of this paper. 

\begin{theorem}\label{product} Let $X,Y$ be two openly factorizable spaces. If the product $X\times Y$ is weakly Lindel\"of, then
\begin{enumerate}
\item the product $X\times Y$ is openly factorizable;
\item each continuous map $f:X\times Y\to Z$ to a Tychonov space $Z$ extends to a continuous map $\bar f:\upsilon X\times\upsilon Y\to\upsilon Z$.
\end{enumerate}
\end{theorem}

\begin{proof}  By Theorem~\ref{spectral}, $X$ is a dense subspace of the limit space $\lim{\mathcal S_X}$ of an open $\w$-spectrum $\mathcal S_X=\{X_\alpha,\pi^\gamma_\alpha,A\}$ such that the restrictions $\pi_\alpha|X:X\to X_\alpha$, $\alpha\in A$, of the limit projections are open and surjective. 
By Proposition~\ref{factor}, the limit space $\lim\mathcal S_X$ is a Hewitt completion of $X$.

By the same reason, the Hewitt completion $\upsilon Y$ of $Y$ can be identified with the limit space $\lim\mathcal S_Y$ of an open $\w$-spectrum $\mathcal S_Y=\{Y_\alpha,p^\gamma_\alpha,B\}$ such that the restrictions $p_\alpha|Y:Y\to Y_\alpha$, $\alpha\in B$, of the limit projections are open and surjective.

On the product $A\times B$ consider the partial order: $(\alpha,\beta)\le(\alpha',\beta')$ if $\alpha\le\alpha'$ and $\beta\le\beta'$.
It is easy to see that the partially order set $A\times B$ is $\w$-directed.
It follows that $X\times Y$ is a subspace of the limit space $\lim \mathcal S_X\times \lim \mathcal S_Y$ of the open $\w$-spectrum 
$$\mathcal S=\{X_\alpha\times Y_\beta,\pi_\alpha^\gamma\times p_\beta^\delta,A\times B\}$$such that for every $(\alpha,\beta)\in A\times B$ the restriction $\pi_\alpha\times p_\beta:X\times Y\to X_\alpha\times Y_\beta$ is open and surjective.
Since the product $X\times Y$ is weakly Lindel\"of, we may apply Proposition~\ref{factor} and Theorem~\ref{spectral} and conclude that the product $X\times Y$ is openly factorizable and $\lim\mathcal S_X\times\lim\mathcal S_Y=\upsilon X\times\upsilon Y$ is a Hewitt completion of $X\times Y$.

Now take any continuous map $f:X\times Y\to Z$ to a second countable space $Z$. By Proposition~\ref{factor}, there is an index $(\alpha,\beta)\in A\times B$ and a continuous map $f_{(\alpha,\beta)}:X_\alpha\times Y_\beta\to Z$ such that $f=f_{(\alpha,\beta)}\circ(\pi_\alpha\times p_\beta)|X\times Y$. Then $\bar f=f_{(\alpha,\beta)}\circ(\pi_\alpha\times p_\beta)$ is a continuous extension of the map $f$ onto the product $\lim\mathcal S_X\times\lim\mathcal S_Y=\upsilon X\times\upsilon Y$.

Finally take any continuous map $f:X\times Y\to Z$ to any Tychonov space $Z$. Identify the Hewitt completion $\upsilon Z$ of $Z$ with a closed subspace of $\IR^\kappa$ for a suitable cardinal $\kappa$. The preceding case insures that the map $f$ extends to a continuous map $\bar f:\upsilon X\times\upsilon Y\to\IR^\w$. It follows that
$$\bar f(\upsilon X\times\upsilon Y)=\bar f(\overline{X\times Y})\subset \overline{f(X,Y)}\subset\overline{Z}=\upsilon Z\subset\IR^\kappa.$$
So $\bar f$ is a continuous map into $\upsilon Z$.
\end{proof}

\section{Some comments and open problems} 

In this section we discuss the relation of the  class of  openly factorizable compact spaces to other known classes of compact spaces and pose some open problems. The survey \cite{Shakh} provided the necessary information on various classes of compact spaces.

We recall that a compact space $X$ is called
\begin{itemize}
\item {\em Dugundji compact} if for each embedding $X\to Y$ to another compact space $Y$ there is a linear positive norm one operator $u:C(X)\to C(Y)$ extending continuous functions from $X$ to $Y$;
\item {\em AE(0)-space} if each continuous map $f:B\to X$ defined on a closed subspace $B$ of a zero-dimensional compact space $A$ can extended to a continuous map $\bar f:A\to X$;
\item {\em openly generated} if  $X$ is homeomorphic to the limit $\lim\mathcal S$ of an open continuous $\w$-complete $\w$-spectrum $\mathcal S=\{X_\alpha,p_\alpha^\gamma,A\}$;
\item {\em dyadic compact} if $X$ is a continuous image of the Cantor cube $\{0,1\}^\kappa$ for some cardinal $\kappa$;
\item {\em $\kappa$-adic} if $X$ is a continuous image of some $\kappa$-metrizable compact space;
\item {\em $\kappa$-metrizable} if $X$ admits a $\kappa$-metric.
\end{itemize}

We recall that a  {\em $\kappa$-metric} on $X$ is a function assigning to each point $x\in X$ and a regular closed set $F\subset X$ a non-negative number  $\rho(x,F)$ so that the following axioms hold:
\begin{enumerate}
\item $\rho(x,F)=0$ if and only if $x\in F$;
\item $\rho(x,F)\ge\rho(x,F')$ for any regular closed sets $F\subset F'$ of $X$ 
\item for any regular closed set $F$ the function $\rho(\cdot,F):x\mapsto\rho(x,F)$ is continuous with respect to the first argument;
\item for any point $x\in X$ and a linearly ordered family $\mathcal F$  of regular closed subsets of $X$, we get $\rho(x,\overline{\cup \mathcal F})=\inf_{F\in\mathcal F}\rho(x,F)$.
\end{enumerate}

By the classical result of Haydon \cite{Haydon}, the classes of Dugundji and AE(0)-compacta coincide. By \cite{Scepin81}, the classes of openly generated and $\kappa$-metrizable compacta coincide. It is well-known that each compact topological group is Dugundji compact. Each Dugundji compact is openly generated and each openly generated compact space of weight $\le\aleph_1$ is Dugundji \cite{Scepin81}. Each $\kappa$-adic compact space has countable cellularity \cite{Scepin81}. The hyperspace $\exp(\{0,1\}^{\aleph_1})$ is openly generated but not Dugundji.

The spectral characterization of openly factorizable spaces from Theorem~\ref{spectral} implies that each openly generated compact space is openly factorizable. The simplest example of an openly factorizable compact space which is not openly generated is the ordinal space $[0,\w_1]$. It is not openly generated because has uncountable cellularity. By the same reason, $[0,\w_1]$ is not $\kappa$-adic.

Thus we have the following chain of implications:

{\small
\begin{picture}(500,70)(10,-10)

\put(5,20){compact} 
\put(0,10){topological} 
\put(10,0){group} 
\put(48,10){$\Ra$}

\put(61,15){Dugundji}
\put(61,5){compact}
\put(103,10){$\Leftrightarrow$}

\put(115,10){AE(0)-compact}
\put(180,10){$\Ra$}
\put(140,25){$\Uparrow$}

\put(130,40){dyadic}
\put(180,40){$\Ra$}

\put(191,10){$\kappa$-metrizable}
\put(246,10){$\Leftrightarrow$}

\put(200,40){$\kappa$-adic}
\put(210,25){$\Uparrow$}

\put(262,15){openly}
\put(258,5){generated}
\put(298,10){$\Ra$}

\put(316,15){openly}
\put(309,5){factorizable}
\end{picture}
}

Let us observe that the classes of openly generated and openly factorizable compact spaces are preserved by open normal functors in the sense of Shchepin \cite{Scepin81}, see also \cite{TZ}. This allows us to construct many openly factorizable compacta failing to be Dugundji compact.

There is another chain of important classes of compact spaces, that is ``orthogonal'' to the chain of classes considered above.

We recall that a compact space $X$ of weight $\kappa$ is
\begin{enumerate}
\item {\em Corson compact} if $X$ embeds into the $\Sigma$-product of lines $$\Sigma=\{(x_\alpha)\in \IR^\kappa:|\{\alpha\in\kappa:x_\alpha\ne0\}|\le\aleph_0\}\subset\IR^\kappa;$$
\item {\em Eberlein compact} if $X$ embeds into the subspace $$\Sigma_0=\{(x_\alpha)\in\IR^\kappa:\forall \varepsilon>0 \;|\{\alpha\in\kappa:|x_\alpha|<\varepsilon\}|<\aleph_0\}\subset \IR^\kappa;$$
\item {\em Valdivia compact} if $X$ embeds into $\IR^\kappa$ so that $X\cap\Sigma$ is dense in $X$.
\end{enumerate}

Those properties relate as follows:
\smallskip

\centerline{Eberlein compact $\Ra$ Corson compact $\Ra$ Valdivia compact.}
\smallskip

Each Eberlein compact with countable cellularity is metrizable \cite{EbC}. So the classes of Eberlien compacta and $\kappa$-adic compacta intersect by the class of metrizable compacta.

\begin{problem} Is each openly factorizable Eberlein (or Corson) compact space metrizable?
\end{problem}

The openly factorizable space $[0,\w_1]$  is known to be Valdivia compact while $[0,\w_2]$ is not Valdivia \cite{Kalenda}. 

\begin{problem} Is each ordinal space $[0,\lambda]$ openly factorizable for each ordinal $\lambda$?
\end{problem} 

The  ordinals segments are examples of both scattered  and linearly ordered compacta. 
 We recall that a topological space $X$ is {\em scattered} if each subspace of $X$ has an isolated point. Scattered spaces need not be openly factorizable. The simplest example is the one-point compactification $\alpha\aleph_1$ of a discrete space of cardinality $\aleph_1$. This space is Eberlein compact but not linearly ordered.

The simplest example of a  linearly orderable scattered compact space that fails to be openly factorizable is the bouquet of the spaces $[0,\w_1]$ and $[0,\w]$ with points $\w_1$ and $\w$ glued together.

\begin{problem} Characterize openly factorizable spaces in the class of scattered (compact) spaces; in the class of linearly ordered (compact) spaces. 
\end{problem}


\newpage

\begin{thebibliography}{}

\bibitem{EbC} A.V.~Arkhangelski, {\em Eberlein compacta}, in: Encyclopedia of general topology. (K.P. Hart, J.~Nagata, J.~Vaughan eds.), Elsevier Sci. Publ., Amsterdam, 2004. -- p.145--146.

\bibitem{AH} A.V.~Arkhangelski, M.~Hu\v sek, {\em Extensions of topological and semitopological groups and product operations}, CMUC, {\bf 42}:1 (2001), 173--186.

\bibitem{BDG} T.~Banakh, S.~Dimitrova, O.~Gutik, {\em Embedding the bicyclic semigroup into countably compact topological semigroups}, preprint (arXiv:0811.4276).

\bibitem{BCF} T.~Banakh, A.~Chigogidze, V.V.~Fedorchuk, {\em On spaces of $\sigma$-additive probability measures}, Topology Appl. {\bf 133}:2 (2003), 139--155. 

\bibitem{FC} A.~Chigogidze, V.V.~Fedorchuk, {\em Absolute retracts and infinite-dimensional manifolds}, Moscow, Nauka, 1992 (in Russian). 

\bibitem{En} R.~Engelking, General Topology, Warsaw, PWN, 1977. 

\bibitem{GFTW} S.~Garcia-Ferreira, A.H.~Tomita, S.~Watson, {\em Countably compact groups from a selective ultrafilter}, Proc.~Amer.~Math.~Soc. {\bf 133} (2005), 937--943.

\bibitem{HS} N.~Hindman, D.~Strauss,  {Algebra in the Stone-Cech compactification. Theory and applications.} de Gruyter Expositions in Mathematics, 27. Walter de Gruyter \&\  Co., Berlin, 1998. 

\bibitem{Haydon} R.~Haydon, {\em On a problem of Pelczynski: Milutin spaces, Dugundji spaces and {\rm AE(0-dim)}}, Studia Math. {\bf 52} (1974), 23--31.

\bibitem{Kalenda} O.~Kalenda, {\em Valdivia compact spaces in topology and Banach space theory}, Extracta Math. {\bf 15}:1 (2000), 1--85.

\bibitem{KTW} P.~Koszmider, A.~Tomita, S.~Watson, {\em Forcing countably compact group topologies on a larger free Abelian group}, Topology proc. {\bf 25} (2000), 563--574.

\bibitem{Megrel} M.~Megrelishvili, {\em Every semitopological semigroup compactification of the group $H\sb +[0,1]$ is trivial}, Semigroup Forum {\bf 63}:3 (2001), 357--370.

\bibitem{MGT} R.~Madariaga-Garcia, A.H.~Tomita, {\em Countably compact topological group topologies on free Abelian groups from selective ultrafilters}, Topology Appl. {\bf 154} (2007), 1470--1480.

\bibitem{PT} V.~Pestov, M.~Tkachenko, {\em Problem 3.28}, in Unsolved Problems of Topological ALgebra, Acad. of Sci. Moldova, Kishinev, ``Shtiinca'' 1985, p.18.

\bibitem{Rez} E.A.~Reznichenko, {\em Extension of functions defined on products of pseudocompact spaces and continuity of the inverse in pseudocompact groups}, Topology Appl. {\bf 59}:3 (1994), 233--244. 

\bibitem{RU} E.A.~Reznichenko, V.V.~Uspenskij, {\em Pseudocompact Mal'tsev spaces},  Topology Appl. {\bf 86}:1 (1998),  83--104.

\bibitem{Rup} W.~Ruppert, {\em Compact semitopological semigroups: an intrinsic theory}, LNM {\bf 1079}, Springer, 1984.

\bibitem{Shakh} D.~Shakhmatov, {\em Compact spaces and their generalizations}, in: Recent progress in general topology (Prague, 1991), 571--640, North-Holland, Amsterdam, 1992. 


\bibitem{Scepin81} E.V.~\v S\v cepin, {\em Functors and uncountable powers of compacta}.  Uspekhi Mat. Nauk 36 (1981), no. 3(219), 3--62.

\bibitem{TZ} A.~Teleiko, M.~Zarichnyi, {Categorical topology of compact Hausdorff spaces}. Monograph Series, {\bf 5}. VNTL Publishers, L'viv, 1999. 

\bibitem{Tka} M.~Tkachenko, {\em Countably compact and pseudocompact topologies on free Abelian groups}, Soviet Math. (Iz. VUZ) {\bf 34}:5 (1990), 79--86.


\bibitem{Tom96} A.H.~Tomita, {\em The Wallace problem: A counterexample from 
MA${}_\mathrm{countable}$ and $p$-compactness}, Canad. Math. Bull {\bf 39}:4 (1996), 486--498.


\end{thebibliography}
\end{document}